\documentclass[11pt]{article}
\usepackage{bm}
\usepackage{float} 
\usepackage[utf8]{inputenc}
\usepackage[dvips]{graphicx}
\usepackage{latexsym}
\usepackage{amsmath,amssymb,amsfonts,amsthm,bbm}
\usepackage[dvipsnames]{xcolor}
\usepackage{enumerate} 
\usepackage[all,poly]{xy}
\usepackage{stmaryrd}
\usepackage{textcomp}
\usepackage{caption,tikz}
\usepackage[colorlinks=true, linkcolor=NavyBlue, urlcolor=black, citecolor=NavyBlue]{hyperref}

\newcommand \NN {\mathbb{N}}
\newcommand \RR {\mathbb{R}}
\newcommand \ZZ {\mathbb{Z}}

\newcommand \EE {\mathbb{E}}
\newcommand \PP {\mathbb{P}}

\newcommand \DP {\overset{\mathrm{def}}{=}}

\newcommand \X   {{\cal X}}
\newcommand{\calP}{\mathcal P}
\newcommand{\calF}{\mathcal F}

\newcommand{\q}{\quad}

\renewcommand{\d}{\delta}

\renewcommand{\O}{\Omega}

\renewcommand{\epsilon}{\varepsilon}

\newcommand{\bpn}{{\pmb p}}
\newcommand{\bPn}{{\pmb P}}
\newcommand{\bp}{{\hskip .1em (\pmb p,\pmb a) }}
\newcommand{\bP}{{\hskip .1em (\pmb P, \pmb A)}}

\theoremstyle{plain}
\newtheorem{theorem}{Theorem}[section]
\newtheorem{proposition}[theorem]{Proposition}
\newtheorem{lemma}[theorem]{Lemma}
\newtheorem{corollary}[theorem]{Corollary}

\theoremstyle{definition}
\newtheorem{definition}[theorem]{Definition}
\theoremstyle{remark}
\newtheorem{remark}[theorem]{Remark}

\let\originalleft\left
\let\originalright\right
\renewcommand{\left}{\mathopen{}\mathclose\bgroup\originalleft}
\renewcommand{\right}{\aftergroup\egroup\originalright}

\author{ Serge Cohen\footnote{ Institut de Math\'ematiques de Toulouse; UMR 5219, Université de Toulouse; CNRS, UT3 F-31062 Toulouse Cedex 9, France. Serge.Cohen@math.univ-toulouse.fr}
\and Shambo Saha \footnote{ Indian Statistical Institute, 203, B T Road, Baranagar, West Bengal, India, shambosaha140704@gmail.com}
}

\title{Split merge dynamics for expanding intervals and point processes on the real line.}

\begin{document}

\maketitle
\begin{abstract}
 We study sequences of partitions of a non decreasing sequence $I_n$ of intervals into subintervals, starting from the trivial partition,
in which each partition is obtained from the one before by splitting its subintervals in two, according to a given rule, 
and then merging pairs of subintervals at the break points of the old partition.
The $n$th partition then comprises $n+1$ subintervals with $n$ break points.  
When $ I_n = [0,1] $ is constant,  the empirical distribution of these points was shown to converge weakly to a singular probability supported in $\{0,1\} $ in a  previous article. When the length of the intervals is regularly varying with a positive index, we show in this article that the limit can be absolutely continuous. In the last part we extend the split merge dynamics to partitions of $ \mathbb R. $ In this case we characterize invariant distributions and show that special instances of split merge dynamics for expanding intervals converge to these invariant measures vaguely in distribution. 
\end{abstract}


\noindent \textit{AMS classification (2020)}: 60J05, 60F05, 60G55.

\section{Introduction}
In this  article we study sequences of partitions $(\calP_n)_{n\ge0}$ of intervals $(a_{n,0},a_{n,n+1}]$, of the form
$$
\calP_n=\{(a_{n,0},a_{n,1}],\dots,(a_{n,n},a_{n,n+1}]\}
$$
where  $(a_{n,0})_{n\geq 0} $ is  a  fixed non increasing sequence  and $ (a_{n,n+1})_{n\geq 0}
$ a fixed non decreasing sequence and the $n$ break points $a_{n,1},\dots,a_{n,n}$ satisfy
$$
a_{n,0} \le a_{n,1}\le\dots\le a_{n,n}\le a_{n,n+1} .
$$
Thus $\calP_0=(a_{n,0},a_{n,n+1}]$ and $\calP_n$ are partitions of $(a_{n,0},a_{n,n+1}]$ into $n+1$ subintervals
$$
(a_{n,0},a_{n,n+1}]=(a_{n,0},a_{n,1}]\cup\dots\cup(a_{n,n},a_{n,n+1}].
$$
Suppose we are given a family of splitting proportions and boundary conditions 
$$
\bp=(p_{n,k}:n\ge1,\,1\le k\le n, \; (a_{n,0})_{n\geq 0}, \; (a_{n,n+1})_{n\geq 0} )
$$ 
with $p_{n,k}\in[0,1]$ for all $k$.
For $n\ge1$, we define the break points of the partition $\calP_n$ recursively by 
\begin{equation}
\label{eqind}
a_{n,k}=p_{n,k}a_{n-1,k-1}+(1-p_{n,k})a_{n-1,k},\q k=1,\ldots,n.
\end{equation}
Thus we { \textit split} each interval $(a_{n-1,k-1},a_{n-1,k}]$ of the partition $\calP_{n-1}$ into two subintervals, 
with proportions $1-p_{n,k}$ on the left and $p_{n,k}$ on the right,
and then we { \textit merge }the resulting subintervals at the break points of $\calP_{n-1}$.
We call this split merge dynamics  \textit{fragmentations with erasure} because each interval of the partition $\calP_{n-1}$ is 
fragmented by adding a new point inside it, but then the break points of $\calP_{n-1}$ are erased.
The cases where $p_{n,k}=0$ or $p_{n,k}=1$ for some $n$ and $k$ are allowed: 
in such cases some intervals of the partition will be empty, but we continue to list them `with multiplicity'.
We use the same formula~\eqref{eqind} for the dynamics of the break points in all cases.

See Figure~\ref{fig:def} for an illustration. 

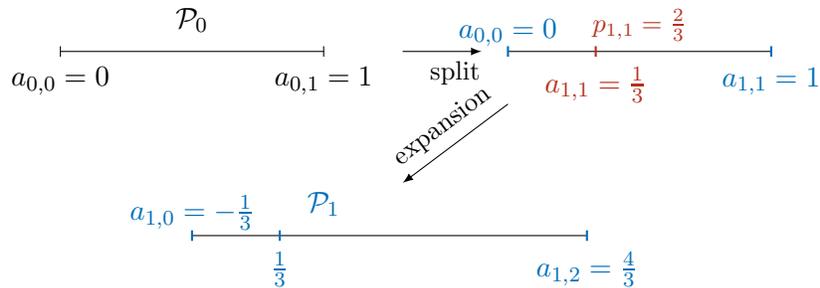
\begin{figure}[H]
	\centering
	\begin{tikzpicture}[scale=3.5]
	\draw (0,0) -- (1,0);
	\draw (0,.02) -- (0,-.02) node[below]{$a_{0,0}=0$};
	\draw (1,.02) -- (1,-.02) node[below]{$a_{0,1}=1$};
	\draw (.5,.12) node{$\calP_0$};
	\draw[->,>=latex] (1.3,0) -- (1.6,0);
	\draw (1.5,0) node[below]{\small split};
	\draw[->,>=latex] (1.7,-.2) -- (1.3,-.5);
	\draw (1.5,-.35) node[above,rotate=37]{\small expansion};
	\begin{scope}[shift={(1.7,0)}]
	\draw (0,0) -- (1,0);
	\draw[thick,NavyBlue] (0,.02) -- (0,-.02) node[above]{$a_{0,0}=0$};
	\draw[thick,BrickRed] (1/3,.02) -- (1/3,-.02) node[below]{$a_{1,1}=\frac{1}{3}$};
	\draw[BrickRed] (.5,.1) node{\small $p_{1,1} = \frac{2}{3}$};
	\draw[thick,NavyBlue] (1,.02) -- (1,-.02) node[below]{$a_{1,1}=1$};
	\end{scope}
	\begin{scope}[shift={(0.5,-.7)}]
	\draw (0,0) -- (1.5,0);
	\draw[thick,NavyBlue] (0,.02) -- (0,-.02) node[above]{$a_{1,0}=-\frac{1}{3}$};
	\draw[thick,NavyBlue] (1/3,.02) -- (1/3,-.02) node[below]{$\frac{1}{3}$};
	\draw[thick,NavyBlue] (1.5,.02) -- (1.5,-.02) node[below]{$a_{1,2}=\frac{4}{3}$};
	\draw[NavyBlue] (.5,.12) node{$\calP_1$};
	\end{scope}
	\end{tikzpicture}
	\caption{Illustration of the first iteration of $(\calP_n)_{n\ge0}$.}
	\label{fig:def}
\end{figure}

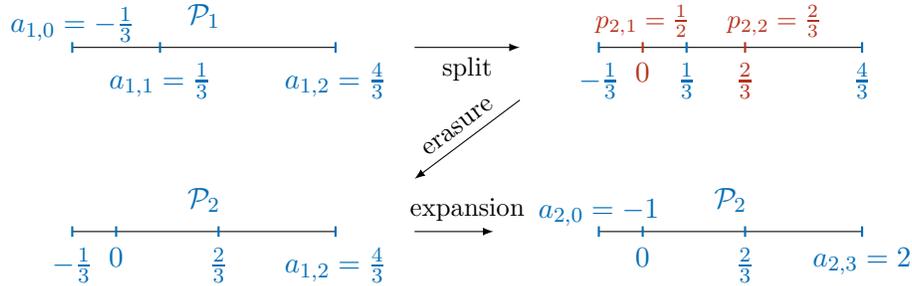
\begin{figure}[H]
	\centering
	\begin{tikzpicture}[scale=3.5]
	
	\begin{scope}
	\draw (0,0) -- (1.0,0);
	\draw[thick,NavyBlue] (0,.02) -- (0,-.02) node[above]{$a_{1,0}=-\frac{1}{3}$};
	\draw[thick,NavyBlue] (1/3,.02) -- (1/3,-.02) node[below]{$a_{1,1}=\frac{1}{3}$};
	\draw[thick,NavyBlue] (1.0,.02) -- (1.0,-.02) node[below]{$a_{1,2}=\frac{4}{3}$};
	\draw[NavyBlue] (.5,.12) node{$\calP_1$};
	\end{scope}
	\begin{scope}
	\draw[->,>=latex] (1.3,0) -- (1.7,0);
	\draw (1.5,0) node[below]{\small split};
	\draw[->,>=latex] (1.7,-.2) -- (1.3,-.5);
	\draw (1.5,-.35) node[above,rotate=37]{\small erasure};
	\end{scope}
	\begin{scope}[shift={(2,0)}]
	\draw (0,0) -- (1,0);
	\draw[thick,NavyBlue] (0,.02) -- (0,-.02) node[below]{$-\frac{1}{3}$};
	\draw[thick,BrickRed] (1/6,.02) -- (1/6,-.02) node[below]{0};
	\draw[BrickRed] (1/6,.1) node{\small $p_{2,1} = \frac{1}{2}$};
	\draw[thick,NavyBlue] (1/3,.02) -- (1/3,-.02) node[below]{$\frac{1}{3}$};
	\draw[thick,BrickRed] (5/9,.02) -- (5/9,-.02) node[below]{$\frac{2}{3}$};
	\draw[BrickRed] (2/3,.1) node{\small $p_{2,2} = \frac{2}{3}$};
	\draw[thick,NavyBlue] (1,.02) -- (1,-.02) node[below]{$\frac{4}{3}$};
	\end{scope}
	\begin{scope}[shift={(0,-0.7)}]
	\draw (0,0) -- (1,0);
	\draw[thick,NavyBlue] (0,.02) -- (0,-.02) node[below]{$-\frac{1}{3}$};
	\draw[thick,NavyBlue] (1/6,.02) -- (1/6,-.02) node[below]{0};
	\draw[thick,NavyBlue] (5/9,.02) -- (5/9,-.02) node[below]{$\frac{2}{3}$};
	\draw[thick,NavyBlue] (1,.02) -- (1,-.02) node[below]{$a_{1,2}=\frac{4}{3}$};
	\draw[NavyBlue] (.5,.12) node{$\calP_2$};
	\end{scope}
	\draw[->,>=latex] (1.3,-0.7) -- (1.6,-0.7);
	\draw (1.5,-.7) node[above]{\small expansion};
	\begin{scope}[shift={(2,-0.7)}]
	\draw (0,0) -- (1,0);
	\draw[thick,NavyBlue] (0,.02) -- (0,-.02) node[above]{$a_{2,0}=-1$};
	\draw[thick,NavyBlue] (1/6,.02) -- (1/6,-.02) node[below]{0};
	\draw[thick,NavyBlue] (5/9,.02) -- (5/9,-.02) node[below]{$\frac{2}{3}$};
	\draw[thick,NavyBlue] (1,.02) -- (1,-.02) node[below]{$a_{2,3}=2$};
	\draw[NavyBlue] (.5,.12) node{$\calP_2$};
	\end{scope}
	
	\end{tikzpicture}
	\caption{Illustration of the second iteration of $(\calP_n)_{n\ge0}$.}
	\label{fig:def2}
\end{figure}

We sometimes choose the family of splitting proportions and of the sequences for the boundary conditions  randomly.
In this case, we will use upper-case letters for the  splitting proportions, the boundary conditions  $\bP=(P_{n,k}:n\ge1,\,1\le k\le n, \;  (A_{n,0})_{n\geq 0}, \; (A_{n,n+1})_{n\geq 0})$
and the interior break points $(A_{n,k} : 1 \le k \le n)$. But their relation will be the same
\begin{equation}
\label{eq:induc-rand2*}
A_{n,k}=P_{n,k}A_{n-1,k-1}+(1-P_{n,k})A_{n-1,k},\q k=1,\ldots,n.
\end{equation}

In this paper, we mainly consider the following three choices for the splitting proportions and the boundary conditions.
\begin{itemize}
\item \textbf{Deterministic stratified fragmentation:} we choose a
deterministic sequence $(p_n)_{n\geq 1}$ in $[0,1]$ and set $p_{n,k}=p_n$ for all $k$. 
An interesting special case is the choice $p_n\equiv p\in(0,1)$. 
\item \textbf{Random stratified fragmentation:} we set $P_{n,k}=P_n$ for all $k$, 
where $(P_n)_{n\ge1}$ are i.i.d.\@ random variables in $[0,1]$. 
\item \textbf{Fully random fragmentation:} we take $(P_{n,k}:n\ge1,\,1\le k\le n)$ to be an 
array of i.i.d.\@ random variables in $[0,1]$, $ (A_{n,0} :n\ge0) $ a non increasing sequence of random variables and $ (A_{n,n+1} :n\ge0) $ a non decreasing sequence of random variables. The three sequences are independent.
\end{itemize}

In each of these three cases, 
the sequence $(\calP_n)_{n\ge0}$ forms a Markov chain whose state-space is the set of finite partitions 
of the unit interval into subintervals.

In a previous article~\cite{Cohen25} we were interested in the behavior 
of the empirical distribution of the break points of the partition $\calP_n$ as $n\to\infty,$   when $ \forall n \ge 0, \; a_{n,0}=0, a_{n,n+1}=1.$  
We set
\begin{equation} 
\label{e:emp.dist}
\tilde{g}_n=\frac1n\sum_{k=1}^n\delta_{a_{n,k}},\q
\tilde{G}_n=\frac1n\sum_{k=1}^n\delta_{A_{n,k}}
\end{equation} 
and we observed the weak convergence of these probabilities to sum of Dirac masses at $0$ or $ 1.$

 In this article  we show that the limit of the empirical probabilities depends on the regular variation index $ c $ of the length $ l_n=a_{n,n+1}-a_{n,0} $  of $(a_{n,0},a_{n,n+1}]$ , when it exists. 
A  new feature appears when $  0 < c < \infty $~:  the limit is absolutely continuous with respect to 
the Lebesgue measure. This fact is coherent with the rule of the thumb that claims that the support 
of the limit in the previous case is $\{0,1\} $ because these point are the only points that are not forgotten in the erasure process. 
In section~\ref{sec:expand} we transport  the empirical distribution $ \tilde{g}_n$ by the affine map $\varphi_n $ from $[a_{n,0}, a_{n,n+1}] $ onto $[0,1].$ When $ 0 < c < \infty,$  we show that  $ (\varphi_n)_* (\tilde{g}_n) $ converges weakly to distributions related to  beta distributions on $ [0,1] $  in sections~\ref{sec:char} and~\ref{sec:weak}.  In section~\ref{sec:gen} we show that, in special instances,   the break points  are asymptotically regularly spaced on the left part and on the right part of $[a_{n,0}, a_{n,n+1}],$ with different spacing  on the left and on the right. An interesting feature described in this article is when  the sequences $(A_{n,0})_{n\geq 0} $   and $ (A_{n,n+1})_{n\geq 0} $  are arrival time of independent Poisson point processes  and the fragmentation with  uniform  splitting proportions. It is studied in section~\ref{sec:rand-expan}. In this case we  prove that at each step the length of the subintervals in the partition are $\Gamma_2$ distributed. It implies that $ G_n = (\varphi_n)_* (\tilde{G}_n)$ converges almost surely to the Lebesgue measure on $[0,1]$ and we have a fluctuation result in this particular case. In Section~\ref{sec:pp} we extend the split merge dynamics to partitions $\calP_n = \cup_{ k \in {\mathbb Z}} (a_{n,k}, a_{n,k+1} ] $ of the full real line $ \mathbb R.$ Such partitions are in one to one correspondence  with integer valued measures $\mu_n = \sum_{ k \in {\mathbb Z}} \delta_{a_{n,k}}. $ We prove for full random fragmentation with uniform splitting that the invariant distributions of the split merge dynamics is a stationary point process $ \mu_{\Gamma_2} = \sum_{ k \in {\mathbb Z}}  \delta_{X_{k}}.$ Moreover we show that the increments $ (X_{k+1} - X_k)_{ k \in {\mathbb Z}} $ are i.i.d. $ \Gamma_2 $ random variables. At last we show that random measures associated to the special instances of section~\ref{sec:rand-expan} converge vaguely in distribution to  $ \mu_{\Gamma_2}.$

\section{Assumptions}
\label{sec:expand}
Let us first precise the assumptions for the sequences
$
(a_{n,0})_{n\geq 0},\ (a_{n,n+1})_{n\geq 0}.
$ 
 In all cases we assume $
(a_{n,0})_{n\geq 0} $ is a non increasing sequence and $  (a_{n,n+1})_{n\geq 0} $ 
a non decreasing function, and let 
\begin{equation}
    \label{eq1}
    l_n=a_{n,n+1}-a_{n,0} \quad \text{and}\quad q_n=\frac{a_{n,n+1}-a_{n-1,n}}{l_n-l_{n-1}} \xrightarrow[n\to\infty]{} q
\end{equation}
where $q \in [0,1].$  We distinguish three cases 
\begin{itemize} 
\item \textbf{Slow 
variations, $ c=0 $~:} $l_n $ is slowly varying and we choose  a non decreasing function  $L(\cdot)$ from $[0,+\infty[ $  to $[0,+\infty[ $  such that  $ L(n)=l_n $  and 
\begin{equation}
    \label{eq2}
    \frac{L(tx)}{L(t)}\xrightarrow[t\to\infty]{}1,
\end{equation}

\item \textbf {Regular 
variations, $ 0 < c < \infty $~:} $l_n $ is regularly  varying sequence with index $ c $ and we choose an increasing  function  $L(\cdot)$ from $[0,+\infty[ $  to $[0,+\infty[ $  such that  $ L(n)=l_n $  and  
\begin{equation}
    \label{eq3}
    \frac{L(tx)}{L(t)}\xrightarrow[t\to\infty]{} x^c. 
\end{equation}
\item \textbf{Fast increments 
, $c = \infty $~: } We assume that there exists a non decreasing function
 $L(\cdot)$  from $[0,+\infty[ $  to $[0,+\infty[ $  such that $ L(n)=l_n $  and
\begin{align}
  \label{eq4}
    \frac{L(tx)}{L(t)}\xrightarrow[t\to\infty]{} 0 &,\; \text{if} \; x < 1,\\
    \frac{L(tx)}{L(t)}\xrightarrow[t\to\infty]{} \infty  &,\; \text{if} \; x >  1.
\end{align}  
\end{itemize}
Properties of regularly varying 
sequences can be found for instance in~\cite{Bingham89}, the fast increasing case 
was also considered in~\cite{Feller50}. 

Define the affine transformation
\begin{equation}
    \label{transform}
    \varphi_n:[0,1]\longrightarrow[a_{n,0},a_{n,n+1}],\quad\varphi_n(x)=a_{n,0}+l_nx
\end{equation}
We are interested in the behavior
of the empirical distribution $ \tilde{g}_n, \; \tilde{G}_n $ of the break points of the partition $\calP_n$ as $n\to\infty$. It turns out that it is easier to study  $ (\varphi_n)_* (\tilde{g}_n) $ and 
we denote by  
\begin{equation}  
 \label{eq:gn}
  g_n= (\varphi_n)_* (\tilde{g}_n)= \frac1n\sum_{k=1}^n\delta_{\varphi_n^{-1} (a_{n,k})}
\end{equation} 
 and 
\begin{equation}
    \label{eq:Gn}
G_n=\frac1n\sum_{k=1}^n\delta_{\varphi_n^{-1}(A_{n,k})}.
\end{equation}
 Actually the study of the empirical distributions on $ [0,1] $ is 
more tractable since we have a probabilistic representation of the break points in this case that we introduce in the next section. 

\section{Characterization of the break points}
\label{sec:char}


Define $I=\{(n,k) : n \geq 1,\ 1 \leq k \leq n\}$. Fix
\[
\bp = \left(p_{n,k} : (n,k)\in I,\;  a_{n,0} : n\ge 0,  a_{n,n+1} : n\ge 0 \right).
\]
with $p_{n,k}\in[0,1]$ for all $n$ and $k$ and $ (a_{n,0})_{n \ge 0}$ non increasing, $ (a_{n,n+1})_{n \ge 0} $ non decreasing. Let $(U_{n,k}:(n,k)\in I), \; (U_n^\nu: n\ge 1),   \; (U_n^\epsilon: n\ge 1) $ be independent  families of independent random variables, all uniformly distributed on $[0,1]$  and set
\begin{equation}
\label{y}
Y_{n,k}=\mathbf{1}_{[U_{n,k}\leq p_{n,k}]}, \; \nu_n = \mathbf{1}_{[U_n^\nu \le q_n]}, \; \epsilon_n = \mathbf{1}_{[U_n^\epsilon \le 1-\frac{l_{n-1}}{l_n}]}
\end{equation}
so that $Y_{n,k}$ is a Bernoulli random variable with success probability $p_{n,k},$ $ \nu_n  $ a Bernoulli random variable with success probability $q_n,$ $ \epsilon_n  $ a Bernoulli random variable with success probability $ 1-\frac{l_{n-1}}{l_n},$ (if $ l_{n-1}= l_n, $ $ \epsilon_n=0$ a.s.). Hence  
\begin{equation}
\label{vardef}
Y_{n,k},\quad \nu_n,\quad  \epsilon_n \quad
n\geq 1,\ 1\leq k \leq n
\end{equation}
 are independent.  Write 
\begin{align*}
    &\Omega_0=\{(y_{n,k})_{n\geq1, \ 1\leq k \leq n}:y_{n,k}\in \{0,1\} \ \forall n\geq1, \ 1\leq k \leq n\}
    \\&\Omega_1=\{(u_{n})_{n\geq1}:u_{n}\in \{0,1\} \ \forall n\geq1\}.
\end{align*}

Denote by $\PP^\bp$ the joint law of $(Y_{n,k})_{n\geq1, \ 1\leq k},  (\nu_{n})_{n\geq1},  (\epsilon_{n})_{n\geq1}$ on $\Omega_0 \times \Omega_1 \times \Omega_1=\Omega $. Define a Markov process $(x_n)_{n\geq0}$ on $\Omega$ with $x_0=0$ and
\begin{equation}
    \label{xn}
    x_n=n\epsilon_n\nu_n+(1-\epsilon_n)(x_{n-1}+y_{n,x_{n-1}+1}).
\end{equation}
The following proposition provides a useful representation of the break points. 
\begin{proposition}
    $\forall n\geq0, \ 1\leq k\leq n$, we have
    \[
    a_{n,k}=\varphi_n\left(\PP^\bp(x_n\leq k-1)\right)
    \]
\end{proposition}
\label{ehh}
\begin{proof}
    $\forall n\geq0, \ 1\leq k\leq n$, define, $\alpha_{n,k}=\PP^\bp(x_n\leq k-1)$. So, for $1\leq k\leq n$,
\begin{align*}
    \alpha_{n,k}&=\PP^\bp(x_n\leq k-1)
    \\&=\PP^\bp(x_n\leq k-1|\epsilon_n=1)\left(1-\frac{l_{n-1}}{l_n}\right)+\PP^\bp(x_n\leq k-1|\epsilon_n=0)\frac{l_{n-1}}{l_n}
    \\&=\PP^\bp(n\nu_n\leq k-1)\left(1-\frac{l_{n-1}}{l_n}\right)+\PP^\bp(x_{n-1}+y_{n,x_{n-1}+1}\leq k-1)\frac{l_{n-1}}{l_n}
    \\&=\PP^\bp(\nu_n=0)\left(1-\frac{l_{n-1}}{l_n}\right)+[\PP^\bp(x_{n-1}+y_{n,x_{n-1}+1}\leq k-1|y_{n,x_{n-1}+1}=1)p_{n,k}
    \\&+\PP^\bp(x_{n-1}+y_{n,x_{n-1}+1}\leq k-1|y_{n,x_{n-1}+1}=0)(1-p_{n,k})]\frac{l_{n-1}}{l_n}
    \\&=(1-q_n)\left(1-\frac{l_{n-1}}{l_n}\right)+[\PP^\bp(x_{n-1}\leq k-2)p_{n,k}+\PP^\bp(x_{n-1}\leq k-1)(1-p_{n,k})]\frac{l_{n-1}}{l_n}
    \\&=(1-q_n)\left(1-\frac{l_{n-1}}{l_n}\right)+\frac{l_{n-1}}{l_n}[p_{n,k}\alpha_{n-1,k-1}+(1-p_{n,k})\alpha_{n-1,k}]
    \\&=\frac{a_{n-1,0}-a_{n,0}}{l_n}+\frac{l_{n-1}}{l_n}[p_{n,k}\alpha_{n-1,k-1}+(1-p_{n,k})\alpha_{n-1,k}]
\end{align*}
Using the definition~\eqref{transform},
\begin{align*}
    \varphi_n(\alpha_{n,k})&=a_{n,0}+l_n\left(\frac{a_{n-1,0}-a_{n,0}}{l_n}+\frac{l_{n-1}}{l_n}[p_{n,k}\alpha_{n-1,k-1}+(1-p_{n,k})\alpha_{n-1,k}]\right)
    \\&=a_{n-1,0}+l_{n-1}[p_{n,k}\alpha_{n-1,k-1}+(1-p_{n,k})\alpha_{n-1,k}]
    \\&=p_{n,k}\varphi_{n-1}(\alpha_{n-1,k-1})+(1-p_{n,k})\varphi_{n-1}(\alpha_{n-1,k})
\end{align*}
Since $a_{n,k}$ satisfies the same recursion~\eqref{eqind}, we have $\varphi_n(\alpha_{n,k})=a_{n,k}$.
\end{proof}
\noindent Consider the more general case where the splitting proportions  and the boundary conditions  $\bP $  are random. We assume moreover that the boundary conditions sequences are slowly varying, regularly varying, or fast increasing almost surely and $Q_n \to Q $ a.s. 
By augmenting our probability space $(\O^*,\calF,\PP^*)$ if necessary, 
we can assume that $\O^*$ supports independent families of independent random variables $(U_{n,k}:(n,k)\in I) , \; (U_n^\nu: n\ge 1),  \; (U_n^\epsilon n\ge 1) $ as above,
which are  moreover independent of $\bP$.
Set 
$$
Y_{n,k}=\mathbf{1}_{[U_{n,k}\le P_{n,k}]}, \;  \; \nu_n = \mathbf{1}_{[U_n^\nu \le Q_n]}, \; \epsilon_n = \mathbf{1}_{[U_n^\epsilon \le 1-\frac{L_{n-1}}{L_n}]}
$$
and define a random process $X=(X_n)_{n\ge0}$ on $\O^*$ by setting $X_0=0$ and then, recursively for $n\ge1$,
$$
X_n=n\epsilon_n\nu_n+(1-\epsilon_n)(X_{n-1}+Y_{n,X_{n-1}+1}).
$$
At this point we need another representation for $ x_n $ in~\eqref{xn}.
\begin{definition}
    Define a random process $(\tau_n)_{n\geq1}$ as
    \begin{equation}
        \tau_n=\begin{cases}
        0&\text{if } \ \forall k\leq n, \ \epsilon_k=0,\\
        \max\{k\leq n : \epsilon_k=1\}&\text{o/w }
    \end{cases}
    \end{equation}
\end{definition}

    This can be used to rewrite the process $(x_n)_{n\geq0}$ as
    \begin{equation}    
        x_n=\tau_n\nu_{\tau_n}+\sum_{k=\tau_n+1}^ny_{k,x_{k-1}+1}
    \end{equation}
    where the sum is $0$ if $\tau_n=n$. In fact, it is useful to define an associated process $(s_n)_{n\geq0}$ with $s_0=0$, and for $n\geq1$,
    \begin{equation}
    \label{sn}
        s_n=\sum_{k=1}^ny_{k,x_{k-1}+1}
    \end{equation}
    so that
    \begin{equation}
    \label{rep}
         x_n=\tau_n\nu_{\tau_n}+s_n-s_{\tau_n}.
    \end{equation}

\noindent
We note the following straightforward fact.

Given a bounded measurable function $\Phi$ on $\{0,1\}^I\times\{0,1\}^\NN\times\{0,1\}^\NN$, define $\phi:[0,1]^I \times [0,1]^{\NN} \times [0,1]^{\NN} \to\RR$ by
$$
\phi \bp =\EE^\bp(\Phi(y,\nu,\epsilon))
$$
where $y=(y_{n,k})_{(n,k)\in I},\ \nu=(\nu_{n})_{n\geq1}, \ \epsilon=(\epsilon_{n})_{n\geq1}$ are the r.v.'s defined in \eqref{y} and \eqref{vardef}.
Then $\phi$ is measurable and we have, almost surely,
\begin{equation}
\label{phiPhi}
\EE^*(\Phi(Y,\nu,\epsilon)|\bP)=\phi \bP .
\end{equation}

\noindent
It will be convenient sometimes to regard \eqref{eqind} and \eqref{e:emp.dist}
as defining measurable functions $a_{n,k} \bp $ and $g_n \bp $ of $\bp$, so we can write
$$
A_{n,k}=a_{n,k} \bP,\q G_n=g_n \bP .
$$
Note that $x=(x_n)_{n\geq0}$ is a function of $y,\nu,\epsilon$. So, on taking $\Phi(y,\nu,\epsilon)=1_{\{x_n\le k-1\}}$, we have $\phi \bp =\varphi_n(a_{n,k} \bp )$ by~\eqref{phiPhi}.
Hence, the random break points $(A_{n,k}:(n,k)\in I)$ associated with the random splitting proportions $\bPn$ satisfy, 
almost surely,
$$
A_{n,k}=a_{n,k} \bP =\varphi_n(\PP^*(X_n\le k-1|\bP))
$$

Because of the representation~\eqref{rep} we need to understand the asymptotic distribution 
of $ (\frac{\tau_n}{n},\nu_n).$ Let us start with a few results concerning $ \tau_n.$

\begin{proposition}
\label{prp} For $ 0 < c < \infty,$ 
    \[\frac{L(\tau_n)}{L(n)} \xrightarrow[n\to\infty]{d}u(0,1)\]
    where $ u(0,1) $ is the uniform distribution on $[0,1].$
\end{proposition}
\begin{proof}

    Firstly, note that $\forall k\leq n$,
    \begin{align*}
        \PP(\tau_n\leq k)&=\PP(\epsilon_{k+1}=0,\dots,\epsilon_n=0)
        \\&=\frac{l_k}{l_{k+1}}\cdot\frac{l_{k+1}}{l_{k+2}}\cdots\frac{l_{n-1}}{l_{n}}\\&=\frac{l_k}{l_{n}}
    \end{align*}
    Now, 
    $\forall t \in [0,1]$, define
    \[
    f_n(t)=\max\{k\leq n:L(k)\leq tL(n)\},
    \]
    which is related to the inverse of $ L.$ We rely on section 7 of~\cite{Bingham89} p 28 to get regular variations of the inverse of $ L.$ Consequently we get 
    \begin{align}
            &\frac{f_n(t)}{n}\xrightarrow[n\to\infty]{}t^{1/c}
    \end{align}
  and 
  \begin{equation}
  \label{amz}
  \frac{L(f_n(t))}{L(n)}\xrightarrow[n\to\infty]{}t.
\end{equation}

 Then
    \begin{align*}
        \PP\left(\frac{L(\tau_n)}{L(n)}\leq t\right)&= \PP\left(\tau_n\leq f_n(t)\right)
        \\&=\frac{l_{f_n(t)}}{l_n}
        \\&=\frac{l_{f_n(t)}}{L(f_n(t))}\cdot\frac{L(n)}{l_n}\cdot\frac{L(f_n(t))}{L(n)}\xrightarrow[n\to\infty]{}t\quad\text{by  \eqref{amz}}.
    \end{align*}
\end{proof}
\begin{corollary}
For $ c >0,$ 
    \label{asinf}
    \[
    \tau_n\xrightarrow[n\to\infty]{a.s.}\infty
    \]
\end{corollary}
\begin{proof}
When $ 0 < c <  \infty,$ 
    $\forall M>0$,
    \begin{align*}
        \mathbb{P}\left(\tau_n\leq M\right)&=\mathbb{P}\left(\frac{L(\tau_n)}{L(n)}\leq \frac{L(M)}{L(n)}\right)
        \\&\leq \left|\mathbb{P}\left(\frac{L(\tau_n)}{L(n)}\leq \frac{L(M)}{L(n)}\right)-\frac{L(M)}{L(n)}\right|+\frac{L(M)}{L(n)}
        \\&\leq \sup_{t\in\mathbb{R}}\left|\mathbb{P}\left(\frac{L(\tau_n)}{L(n)}\leq t\right)-t\right|+\frac{L(M)}{L(n)}\xrightarrow[n\to\infty]{} 0\quad \text{by Pólya's theorem.}
    \end{align*}
    Since $ (\tau_n)_{n\geq1}$ is a monotonic sequence,
    \begin{align*}
    &\tau_n\xrightarrow[n\to\infty]{a.s.}\infty
    \end{align*}
    If $ c = + \infty, \lim_{ t \to \infty} L(t) = \infty.$ and  $\forall M>0$
    $ \mathbb{P}\left(\tau_n\leq M\right )=   \frac{L([M])}{L(n)} \xrightarrow[n\to\infty]{} 0,  $ 
    where $ [M] $ is the notation for the integer  part of $M,$
    and the end of the proof is the same. 
\end{proof}
Let us introduce some notations for beta distributions. 
\begin{definition}
For $ r,s > 0, $ a random variable has a distribution $\beta_{r, s}(0,1)$ if its density 
with respect to the Lebesgue measure  $ dx $ on $ [0,1]$ is 
$$ \beta_{r, s}(0,1)(x)=  \frac{x^{r-1} (1-x)^{s-1}}{\int_0^1 x^{r-1} (1-x)^{s-1} dx}{\mathbf 1}_{x \in [0,1]}.$$ For $ a <  b $ we denote $ Y \overset{d}{=} \beta_{r, s}(a,b) $
if $ \frac{Y - a }{b- a } $ has $\beta_{p, q}(0,1)$ distribution. We often drop 
$ (0,1) $    in $\beta_{p, q}(0,1).$ The density of $ \beta_{r, s}(a,b) $ is $ (b-a) \beta_{r, s}(0,1)(\frac{x-a}{y-b}).$  By convention we set $ \beta_{r, s}(a,a)= \delta_a (dx). $

\end{definition}
 Please note the abuse of notation, where $\beta_{r, s}(0,1)$  stands for the distribution 
 and the density of this distribution with respect of $ dx. $ Later we abuse this notation
 for other density functions. 

\begin{corollary}
\label{dinf}
\begin{itemize}
\item If  $ 0 < c < \infty, $ $ \frac{\tau_n}{n} $ converges in distribution~: 
    \[
    \frac{\tau_n}{n} \xrightarrow[n\to\infty]{d} \beta_{(c,1)}.
        \]
\item  If $ c=0, $ $ \frac{\tau_n}{n} $ converges in probability to $0.$ 
\item  If $ c=\infty,  $ $ \frac{\tau_n}{n} $ converges in probability to $1.$  
\end{itemize}
\end{corollary}
\begin{proof} For  $ 0 < c < \infty, $ the statement is equivalent to, $\forall t\in [0,1]$,
    \[
    \mathbb{P}\left(\frac{\tau_n}{n}\leq t\right)\xrightarrow[n\to\infty]{}t^c
    \]
    Consider,
    \begin{align*}
        \left|\mathbb{P}\left(\frac{\tau_n}{n}\leq t\right)-t^c\right|&=\left|\mathbb{P}\left(\tau_n\leq nt\right)-t^c\right|
        \\&=\left|\PP\left(\frac{L(\tau_n)}{L(n)}\leq\frac{L(nt)}
        {L(n)}\right)-t^c\right|
        \\&\leq\left|\PP\left(\frac{L(\tau_n)}{L(n)}\leq\frac{L(nt)}
        {L(n)}\right)-\frac{L(nt)}
        {L(n)}\right|+\left|\frac{L(nt)}
        {L(n)}-t^c\right|
        \\&\leq\sup_{t\in\mathbb{R}}\left|\mathbb{P}\left(\frac{L(\tau_n)}{L(n)}\leq t\right)-t\right|+\left|\frac{L(nt)}
        {L(n)}-t^c\right|
        \xrightarrow[n\to\infty]{} 0
    \end{align*}
  by Pólya's theorem  which gives the result.
    
     For $ c = 0 $ $$ \mathbb{P}\left(\frac{\tau_n}{n}\leq t\right) =\frac{L(nt)}{L(n)}. 
     $$
     and slow variations yields convergence in distributions which is equivalent to convergence in probability. 
     
      For $ c = \infty $ fast increasing assumption yields the result. 
     
\end{proof}
\begin{proposition}
\label{prop:asym-ind}
If  $ 0 \le c \le + \infty,$ 
    $\tau_n/n$ and $\nu_{\tau_n}$ are asymptotically independent~:
    \begin{equation}
    \label{eq:asym-ind}
    \left (\frac{\tau_n}{n} , \nu_{\tau_n} \right ) \xrightarrow[n\to\infty]{d} \left ( X, \epsilon_q \right ) 
    \end{equation}
    where $ X \overset{d}{=} \beta_{c,1}, $ if $  0 < c < + \infty,$ $ X = 0, $ if $ c = 0,$ $ X= 1 $ if $ c = 1, $  and is independent of $ \epsilon_q $ a Bernoulli random variable with parameter $q.$ 
\end{proposition}
\begin{proof}

    We have, $q_n\xrightarrow[n\to\infty]{}q$. For $x\in\{0,1\}$, define $p_n(x)=q_n^x(1-q_n)^{1-x}, \ p(x)=q^x(1-q)^{1-x}$ and note that $|p_n(x)-p(x)|=|q_n-q|$. Then, given $\varepsilon>0$, $\exists N\in\mathbb{N}$ s.t. $\forall n\geq N$, $|q_n-q|<\varepsilon$. Fix $t \in [0,1],\ x \in \{0,1\}$. $
     \forall n$ s.t. $[nt]>N$,
\begin{align*}
    \mathbb{P}\left(\frac{\tau_n}{n}\leq t,\nu_{\tau_n}=x\right)
    &=\mathbb{P}\left(\tau_n\leq [nt],\nu_{\tau_n}=x\right)
    \\&=\sum_{k=0}^{[nt]}\mathbb{P}\left(\tau_n=k,\nu_{\tau_n}=x\right)
    \\&=\sum_{k=0}^{[nt]}\mathbb{P}\left(\tau_n=k\right)p_k(x)   \\&=\sum_{k=0}^{[nt]}\mathbb{P}\left(\tau_n=k\right)p(x)+\sum_{k=0}^{[nt]}\mathbb{P}\left(\tau_n=k\right)(p_k(x)-p(x))
    \\&=\mathbb{P}\left(\frac{\tau_n}{n}\leq t\right)p(x)+\sum_{k=0}^{N}\mathbb{P}\left(\tau_n=k\right)(p_k(x)-p(x))
    \\&+\sum_{k=N+1}^{[nt]}\mathbb{P}\left(\tau_n=k\right)(p_k(x)-p(x))
\end{align*}
which gives,
\begin{align*}
    \left|\mathbb{P}\left(\frac{\tau_n}{n}\leq t,\nu_{\tau_n}=x\right)-\mathbb{P}\left(\frac{\tau_n}{n}\leq t\right)p(x)\right|&\leq \sum_{k=0}^{N}\mathbb{P}\left(\tau_n=k\right)|q_k-q|\\
    &\phantom{\mathbb{P}\left(\frac{\tau_n}{n}\leq t,\right) }+\sum_{k=N+1}^{[nt]}\mathbb{P}\left(\tau_n=k\right)|q_k-q|\\
    &\leq \sum_{k=0}^{N}\mathbb{P}\left(\tau_n=k\right)+\sum_{k=N+1}^{[nt]}\mathbb{P}\left(\tau_n=k\right)\varepsilon
    \\&=\mathbb{P}\left(\tau_n\leq N\right)+\varepsilon
\end{align*}
If $ 0 < c < \infty, $  by Corollary \ref{asinf} and \ref{dinf}, $ \forall t \in [0,1],$ 
\begin{align*}
    &\limsup_{n\to\infty}\left|\mathbb{P}\left(\frac{\tau_n}{n}\leq t,\nu_{\tau_n}=x\right)-\mathbb{P}\left(\frac{\tau_n}{n}\leq t\right)p(x)\right|\leq\varepsilon\quad\forall\varepsilon>0
    \\\implies&\lim_{n\to\infty}\mathbb{P}\left(\frac{\tau_n}{n}\leq t,\nu_{\tau_n}=x\right)=\lim_{n\to\infty}\mathbb{P}\left(\frac{\tau_n}{n}\leq t\right)p(x)=t^cp(x)
\end{align*}
If $ c= 0,$ $ \lim_{n\to\infty}\mathbb{P}\left(\frac{\tau_n}{n}\leq t\right)p(x)=p(x), $ 
and if $ c = \infty, $ $ \lim_{n\to\infty}\mathbb{P}\left(\frac{\tau_n}{n}\leq t\right)p(x)= {\mathbf 1}_{t=1} p(x).$ 
\end{proof}
\section{Weak convergence of the empirical distributions}
\label{sec:weak}
In this section we prove the weak convergence of the empirical distributions in Theorem~\ref{thm:cv_p} which is a consequence of the limit in distribution of $ \frac{x_n}{n}.$ 
\begin{lemma}
\label{lem:LLN}
Fix a family of splitting proportions and boundary conditions $\bpn=(p_{n,k}:(n,k)\in I, a_{n,0}, a_{n,n+1}:n \in \mathbb{N})$ 
and consider the associated random process $(s_n)_{n\ge0}$ defined on $\O$ by \eqref{sn}.
    Suppose, $\exists \bar{p}\in[0,1]$ s.t. under $\PP^\bp$,
    \begin{equation}
    \label{cond}
        \frac{s_n}{n}\xrightarrow[n\to\infty]{P}\bar p
    \end{equation}
Then, as $n\to\infty$,
\[
g_n \xrightarrow[n\to\infty]{d} \begin{cases}
    \bar{p}\beta_{\left(1,\frac{1}{c}\right)}(0,1-q)+(1-\bar{p})\beta_{\left(\frac{1}{c},1\right)}(1-q,1) &\text{if }\ 0<c<\infty
    \\\bar{p}\d_0+(1-\bar{p}) \d_1 &\text{if }\ c=0
    \\\d_{1-q} &\text{if }\ c=\infty
\end{cases}
\; \text{weakly on }[0,1].
\]
\end{lemma}
\begin{proof}
Note that, $\forall t \in (0,1)$,
    \begin{align*}
        g_n([0,t])&=\frac{1}{n} \max\{k:\alpha_{n,k}\leq t\}
        \\&=\frac{1}{n} \max\{k:\PP^\bp(x_n\leq k-1)\leq t\}
        \\&=\frac{1}{n} \max\{k:\PP^\bp\left(x_n/n\leq (k-1)/n\right)\leq t\}.
     \end{align*}
     
Since  $F_n^\bp(x)=\PP^\bp\left(x_n/n\leq x\right)$ is constant on $ [\frac{k-1}{n},\frac{k}{n}) $ for $ k =1 $ to $ n, $ if $ k_0= \max\{k:F_n^\bp(\frac{k-1}{n}) \le t \}, $ 
$ \sup\{x: F_n^\bp(x) \le t \} = \frac{k_0+ 1}{n} $
and $$  g_n([0,t]) = \sup\{x: F_n^\bp(x) \le t \}  - \frac{1}{n}. $$ 
Hence to prove weak convergence of $ g_n $ it is enough to show convergence in distribution of $ \frac{x_n}{n}.$ Actually if $ F_n^\bp(x) \to F $ for every continuity point of $ F, $  $g_n([0,t])\rightarrow Q(t)$, where $Q$ is the right-continuous quantile function for $F.$
        

Let us show the convergence in distribution of $ \frac{x_n}{n} $ using the representation \eqref{rep},
    \begin{align}
    \frac{x_n}{n}&= \frac{\tau_n}{n}  \nu_{\tau_n}+\frac{s_n-s_{\tau_n}}{n} \notag 
    \\&=\frac{\tau_n}{n}  \nu_{\tau_n}+ \frac{s_n}{n} - \frac{\tau_n}{n} \frac{s_{\tau_n}}{\tau_n} \label{eq:frac}
\end{align}
If $ 0 < c \le  \infty,$  under condition \eqref{cond}, using Corollary \ref{asinf}, Proposition~\ref{prop:asym-ind},    the fact that $\tau_n/n \le 1$, and Slutsky lemma
\begin{equation}
\label{whoa}
     \frac{x_n}{n} \xrightarrow{d} X \epsilon_q + \bar p - X \bar p. 
\end{equation}
The distribution of $ X \epsilon-q + \bar p - X \bar p$ given $ \epsilon_q= 1$  is $ \beta_{c,1}(\bar p,1) $ and given $ \epsilon_q= 0$ it is  $  \beta_{1,c}(0,\bar p). $ Hence $ \frac{x_n}{n} $ converges to a distribution with density $$ (1-q)\beta_{(1,c)}(0,\bar{p})+q\beta_{(c,1)}(\bar{p},1) $$ when $ \bar p \in (0,1).$ When  $ \bar p \in \{0,1\}, $ the limit may include an atom, which is specified by $ \beta_{r,s}(a,a) = \delta_a.$ The computation of the quantile $ Q $ yields Lemma~\ref{lem:LLN} when $ c >0.$



    For $ c=0,$ Proposition~\ref{prop:asym-ind}~ yields $ \frac{\tau_n}{n} \xrightarrow{P} 0.$ Moreover~\eqref{cond} and $ 0 \le s_{\tau_n} \le s_n, $ yields $ \frac{s_n}{n} - \frac{\tau_n}{n} \frac{s_{\tau_n}}{\tau_n} \xrightarrow{P} \bar p,$  hence  $$ \frac{x_n}{n} \xrightarrow{P} \bar p $$ 
    which yields  Lemma~\ref{lem:LLN} when $ c = 0.$ 


\end{proof}
\begin{theorem}	\label{thm:cv_p}
\textup{(a)} {\bf Deterministic stratified fragmentation}. 
Assume that the averages $\frac1n\sum_{k=1}^np_k$ converge as $n\to\infty$, with limit $\bar p$ say.
Then, as $n\to\infty$, 
$$ 
g_n \xrightarrow[n\to\infty]{d} \begin{cases}
    \bar{p}\beta_{\left(1,\frac{1}{c}\right)}(0,1-q)+(1-\bar{p})\beta_{\left(\frac{1}{c},1\right)}(1-q,1) &\text{if }\ 0<c<\infty
    \\\bar{p}\d_0+(1-\bar{p}) \d_1 &\text{if }\ c=0
    \\\d_{1-q} &\text{if }\ c=\infty
\end{cases}
\; \text{weakly on }[0,1].
$$
\textup{(b)} {\bf Random stratified fragmentation and fully random fragmentation}. 
Define $\bar p=\EE(P_1)$ or $\bar p=\EE(P_{1,1})$ according to the case in hand.
Then, almost surely, as $n\to\infty$, 
$$
G_n \xrightarrow[n\to\infty]{d}   \begin{cases}
    \bar{p}\beta_{\left(1,\frac{1}{c}\right)}(0,1-q)+(1-\bar{p})\beta_{\left(\frac{1}{c},1\right)}(1-q,1) &\text{if }\ 0<c<\infty
    \\\bar{p}\d_0+(1-\bar{p}) \d_1 &\text{if }\ c=0
    \\\d_{1-q} &\text{if }\ c=\infty
\end{cases}
\;\text{weakly on }[0,1].
$$
\end{theorem}
Please remark that in the special case $ c=1, $ the beta distributions are actually uniform distributions. 
Moreover if $ q= \bar{p} = \frac12, $ the empirical distributions of the break points converges to a uniform distribution on $[0,1].$ 
\begin{proof}
For part (a), under $\PP^\bp$, the process $(s_n)_{n\ge0}$ is a sum of independent Bernoulli random variables, 
with success probabilities $p_1,p_2,\dots$
Then, by Hoeffding's inequality, for all $\epsilon>0$, 
\begin{align*}
    &\PP^\bp\left(\left|\frac{s_n}n-\frac1n\sum_{k=1}^np_k\right|\ge\epsilon\right)
\le2\exp\left\{-2n\epsilon^2\right\}
\\\implies&\frac{s_n}{n}\xrightarrow[n\to\infty]{P}\bar p
\end{align*}
and Lemma \ref{lem:LLN} applies.
\\For part (b), note that, in both cases,  under $ \PP^*,$ 
the process $(S_n)_{n\ge0}$ is a random walk with mean step size $\bar p$,
so $S_n/n\to\bar p$ almost surely by the strong law of large numbers.
Set
$$
\Phi(y,\nu,\epsilon)=1_{\{s_n/n\to\bar p\text{ as }n\to\infty\}}
$$
then
$$
\phi(\bp)=\EE^\bp(\Phi(y,\nu,\epsilon))=\PP^\bp(s_n/n\to\bar p\text{ as }n\to\infty).
$$
By~\eqref{phiPhi}, almost surely,
$$
\phi \bP =\PP^*(S_n/n\to\bar p\text{ as }n\to\infty|\bP).
$$
Hence $\phi(\bP)=1$ almost surely. Consider the event $\tilde \O=\{\bPn\in B\}$, where
$$
B=\{\bp:\frac{s_n}{n}\xrightarrow[n\to\infty]{P}\bar p \text{ under }\PP^\bp\}.
$$
Since almost sure convergence implies in probability convergence,
\[
\PP^*(\tilde \O)\geq\PP^*(\phi(\bP)=1)=1
\]
By Lemma \ref{lem:LLN}, for $\bp\in B$, we have the weak limit for $g_n(\bp)$ as $n\to\infty$.
Then, it is clear that, 
\[G_n \xrightarrow[n\to\infty]{d}  \begin{cases}
    \bar{p}\beta_{\left(1,\frac{1}{c}\right)}(0,1-q)+(1-\bar{p})\beta_{\left(\frac{1}{c},1\right)}(1-q,1) &\text{if }\ 0<c<\infty
    \\\bar{p}\d_0+(1-\bar{p}) \d_1 &\text{if }\ c=0
    \\\d_{1-q} &\text{if }\ c=\infty
\end{cases}
\; \text{weakly on }[0,1].\]

\end{proof}

 If $ c= 0 $ we may think of Theorem~\ref{thm:cv_p} as a generalization of Theorem 3.2 in~\cite{Cohen25} when $ l_n$ is not constant but slowly varying. If $ 0 < c < \infty $ and $ 0 < q < 1 $  a new feature appears since the limit distribution has a density with respect to Lebesgue measure. Let us also remark that Theorem~\ref{thm:cv_p} yields partial information on the limit distribution of the $(a_{n,k})_{0 \le k \le n+1} $ on the real line. It is a consequence of the fact that when $ c >0 $ $ \varphi_n $ is $l_n$-Lipschitz 
 and $ l_n \to \infty.$ In some cases we can be more precise.
 \begin{lemma}
 \label{lem:reg}
  Let us assume $ \forall n, \; k, \; p_{n,k} = p \in (0,1), $ and, for $ r > 0, $ take $ a_{n,0}= -r p (n+1) $ and $ a_{n,n+1}= r (1-p) (n+1).$ Then  $$ \forall n \ge 1, \forall 0 \le k \le n+1, a_{n,k} = r ( - p (n+1) + k).$$ 
 \end{lemma}
 It is straightforward to  check the Lemma by induction. This setting will be called the regular case. 
 It  is consistent with Theorem~\ref{thm:cv_p}. Here $l_n $ is regularly varying with index $ c = 1 $ and $ q_n= 1 - p.$ Actually in this case $ \alpha_{n,k}= \varphi_n^{-1}(a_{n,k}) = \frac{k}{n+1} $ which is  more accurate than  the convergence of   Theorem~\ref{thm:cv_p}. The next section investigates what happens when we introduce a small  perturbation of this regular case. 
 
\section{Generating function approach}
\label{sec:gen}

Let us assume  $ \forall n, \; k, \; p_{n,k} = p \in (0,1), $ $ a_{n,0}= - \frac{n+1}{2},  $ and $ a_{n,n+1}=  \frac{n+1}{2}. $ Hence $ l_n = n+1, q_n = \frac12 $ and $ p = \frac12$ is nothing else than the regular case for $ r = 2. $ In this section we consider also the case $ p \neq \frac12, $ which introduces asymmetry in the splitting of the intervals. The consequence of this asymmetry is the following proposition, where we observe a difference in the behavior of the break points on the left and on the right of $[0,1].$

\begin{proposition}
\label{prop:perturb-reg} If  $\forall (n,k)\in I \; p_{n,k} = p \in (0,1), $ $ a_{n,0}= - \frac{n+1}{2},  $ and $ a_{n,n+1}=  \frac{n+1}{2} $ then $ \forall k \in \NN,$
\begin{equation}
\label{eq:dl1}
\alpha_{n,k+1} = \alpha_{n,k} + \frac{1}{2p} + o \left ( \frac{1}{n+1} \right )
\end{equation}
and 
\begin{equation}
\label{eq:dl2}
\alpha_{n,n -k +1 } = \alpha_{n,n-k} + \frac{1}{2(1-p)} + o \left ( \frac{1}{n+1} \right ).
\end{equation}
\end{proposition}

 The Proposition is consistent with Theorem~\ref{thm:cv_p}, here $ c =1,\; \bar{p}=p $ but more precise since it claims that the break points close to $0$ or $ 1$ are in arithmetic progression up to $  o \left ( \frac{1}{n+1} \right ).$ This result cannot be true uniformly in $ k $ and we postpone the study of the intriguing behavior of $ \alpha_{n,k} $ when $ k \sim [\frac{n}{2}] $  to a further study. 
\begin{proof}

In this setting  $ (Y_n)_{n \in \mathbb N}  $ are i.i.d. Bernoulli random variables with parameter $ p,$ 
$ (\nu_n)_{n \in \mathbb N}  $ are i.i.d. Bernoulli random variables with parameter $ 1/2,$ $\epsilon_n$  i.i.d. Bernoulli random variables with parameter $ \frac{1}{n+1},$ all independent, 
the  Markov chain~\eqref{xn} $(X_n)_{n\geq0}$  can be written  $X_0=0$ and
\(
X_n=n\epsilon_n \nu_n +(1-\epsilon_n)(X_{n-1}+Y_{n}).
\)
It has been previously shown that
\[
\alpha_{n,k}=\mathbb{P}(X_n\leq k-1)=\frac{1}{2(n+1)}+\frac{n}{n+1}[p\alpha_{n-1,k-1}+(1-p)\alpha_{n-1,k}]
\]
Define $H_n(k)=(n+1)\mathbb{P}(X_n= k)$.
\begin{align}
     H_n(k)&=(n+1)(\alpha_{n,k+1}-\alpha_{n,k}) \notag
    \\&=(n+1)\frac{n}{n+1}[p(\alpha_{n-1,k}-\alpha_{n-1,k-1})+(1-p)(\alpha_{n-1,k+1}-\alpha_{n-1,k})] \notag
    \\&=pH_{n-1}(k-1)+(1-p)H_{n-1}(k) \quad \forall1\leq k \leq n-1 \label{Hnk}
    \\H_n(0)&=(n+1)(\alpha_{n,1}-\alpha_{n,0}) \notag
    \\&=(n+1)\alpha_{n,1} \notag
    \\&=\frac{n+1}{2(n+1)}+\frac{n(n+1)}{n+1}[p\alpha_{n-1,0}+(1-p)\alpha_{n-1,1}] \notag
    \\&=\frac{1}{2}+(1-p)n\alpha_{n-1,1} \notag
    \\&=\frac{1}{2}+(1-p)H_{n-1}(0) \label{Hn0}
    \\H_n(n)&=(n+1)(\alpha_{n,n+1}-\alpha_{n,n}) \notag
    \\&=(n+1)(1-\alpha_{n,n}) \notag
    \\&=(n+1)-\frac{n+1}{2(n+1)}-\frac{n(n+1)}{n+1}[p\alpha_{n-1,n-1}+(1-p)\alpha_{n-1,n}] \notag
    \\&=\frac{1}{2}+pn(1-\alpha_{n-1,1}) \notag 
    \\&=\frac{1}{2}+pH_{n-1}(n-1). \label{Hnn}
\end{align}
Now, let $\mathcal{H}_n(s)=\sum_{k=0}^nH_n(k)s^k$. Then,
\begin{align*}
    \mathcal{H}_n(s)&=\sum_{k=0}^nH_n(k)s^k
    \\&=H_n(0)+\sum_{k=1}^{n-1}H_n(k)s^k+H_n(n)s^n
    \\&=\frac{1}{2}+(1-p)H_{n-1}(0)+\sum_{k=1}^{n-1}[pH_{n-1}(k-1)+(1-p)H_{n-1}(k)]s^k\\
    &\phantom{\frac{1}{2}+(1-p)H_{n-1}(0)+}+[\frac{1}{2}+pH_{n-1}(n-1)]s^n
    \\&=\frac{1}{2}(1+s^n)+(1-p)\sum_{k=0}^{n-1}H_{n-1}(k)s^k+ps\sum_{k=0}^{n-1}H_{n-1}(k)s^k
    \\&=\frac{1}{2}(1+s^n)+[(1-p)+ps]\mathcal{H}_{n-1}(s).
\end{align*}
Now,
\[
\frac{1}{2p(1-s)}=\frac{1}{2}+[(1-p)+ps]\frac{1}{2p(1-s)}.
\]
So, let
\[
\mathcal{C}_n(s)=\mathcal{H}_n(s)-\frac{1}{2p(1-s)}.
\]
Then,
\begin{align*}
    &\mathcal{C}_n(s)\\&=\frac{1}{2}s^n+[(1-p)+ps]\mathcal{C}_{n-1}(s)
    \\&=\frac{1}{2}\sum_{k=0}^{n-1}s^{n-k}[(1-p)+ps]^k+[(1-p)+ps]^n\mathcal{C}_0(s)
    \\&=\frac{(s^n-[(1-p)+ps]^n)}{2(1-\frac{[(1-p)+ps]}{s})}+[(1-p)+ps]^n\mathcal{C}_0(s) \xrightarrow[n\to\infty]{}0 \quad \forall|s|<1
    \\& \implies \mathcal{H}_n(s)\xrightarrow[n\to\infty]{}\frac{1}{2p(1-s)}\quad \forall|s|<1.
\end{align*}
So that, $\forall k\in\mathbb{N}$,
\[
H_n(k)\xrightarrow[n\to\infty]{}\frac{1}{2p}. 
\]
One can also show that 
$\forall k\in\mathbb{N}$,
\[
H_n(n-k)\xrightarrow[n\to\infty]{}\frac{1}{2(1-p)}
\]
by the symmetry $ k \mapsto n-k, $ $ p \mapsto 1-p $ in the equations~\eqref{Hnk}~\eqref{Hn0}~\eqref{Hnn}.
\end{proof}
\section{Full Random Case (Special setup)}
\label{sec:rand-expan}
In this section we allow $ A_{n,0} $ and $ A_{n, n+1} $ to be random and we consider 
a special case in the same spirit as in the previous proposition. 

Let us first introduce some notations for exponential and gamma distributions. 
\begin{definition}
For $ p >0, \; \theta >0, $ a random variable has a distribution $\Gamma_{p, \theta} $ if its density 
with respect to the Lebesgue measure  $ dx $ on $ \RR^+$ is 
$$ \gamma_{p, \theta}(x)=  \frac{\theta^p e^{-\theta x} x^{p-1}{\mathbf 1}_{x \in [0,+\infty)}}{\Gamma(p)}.$$
For $ \theta =1,$ we shorten the notation to $\Gamma_{p}.$ Classicaly for $ p = 1 $ $ \Gamma_{1, \theta} $
are exponential random variables with parameter $ \theta.$ 
\end{definition}

We begin with the following collections of random variables,
\[
(E_n)_{n\geq0}, \ (E'_n)_{n\geq0}
\]
two independent sequences of i.i.d. exponential random variables with parameter $1 $ 
defined on $ \Omega^* $. These sequences are assumed be independent of 
 the  $(U_{n,k}) $ and of the $(P_{n,k}) $ which are moreover assumed to be uniformly distributed on $ [0,1].$  
Define
\[
S_n=\sum_{k=0}^nE_k,\quad S'_n=\sum_{k=0}^nE'_k
\]
Now, the breakpoints for the partitions are defined as
\begin{align}
    &A_{n,0}=-S'_n,\quad A_{n,n+1}=S_n, \quad \forall n\geq0 \label{def:point-rand-expan}
    \\&A_{n,k}=P_{n,k}A_{n-1,k-1}+(1-P_{n,k})A_{n-1,k}, \quad \forall n\geq1, \ 1\leq k\leq n. \notag
\end{align}

Let us also denote by 
\begin{equation}
\label{eq:rand-mea}
\tilde{\mu}_n = \sum_{k=1}^n \delta_{A_{n,k}}. 
\end{equation}
that will be studied in section~\ref{sec:pp}.

Let $I_{n,k}=A_{n,k+1}-A_{n,k}, \ \forall n\geq0, \ 0\leq k\leq n$.

\begin{proposition}
Almost surely under $\PP^* $  $\forall n\geq0,\quad (I_{n,k})_{0\leq k \leq n} $ are independent $\Gamma_2, $ random variables. 

\end{proposition}
\begin{proof}
  The proof is by induction. For $n=0$, \[I_{0,0}=A_{0,1}-A_{0,0}=E_0+E'_0\]
is $ \Gamma_2 $ distributed. Suppose the statement is true for some $n\geq0$. Then for $n+1$,
\begin{align*}
    I_{n+1,k}&=A_{n+1,k+1}-A_{n+1,k}
    \\&=A_{n+1,k+1}-A_{n,k}+A_{n,k}-A_{n+1,k}
    \\&=(1-P_{n+1,k+1})(A_{n,k+1}-A_{n,k})+P_{n+1,k}(A_{n,k}-A_{n,k-1})
    \\&=(1-P_{n+1,k+1})I_{n,k}+P_{n+1,k}I_{n,k-1}.
\end{align*}
By the induction hypothesis, $ (I_{n,k})_{0\leq k \leq n} $ $ \Gamma_2 $ distributed
    implies $$ (P_{n+1,k+1}I_{n,k})_{0\leq k \leq n}, \; ((1-P_{n+1,k+1})I_{n,k})_{0\leq k \leq n} $$
    are two independant sequences of i.i.d. exponential random variables with parametrer $1.$ 

Hence,
\begin{align*}
    (I_{n+1,k})_{0\leq k \leq n+1}=(1-P_{n+1,k+1})I_{n,k}+P_{n+1,k}I_{n,k-1})_{0\leq k \leq n+1}
\end{align*} 
are independent $\Gamma_2, $ random variables. 
By the Principle of Mathematical Induction, the proof is complete.
\end{proof}

In this special setting  the affine transformation is defined by 

\[
\varphi_n:[0,1]\longrightarrow[A_{n,0},A_{n,n+1}],\quad\varphi_n(x)=A_{n,0}+(A_{n,n+1}-A_{n,0})x
\]
and  
\begin{align*}
    \alpha_{n,k}&=\varphi_n^{-1}(A_{n,k})
    \\&=\frac{A_{n,k}-A_{n,0}}{A_{n,n+1}-A_{n,0}}
    \\&=\frac{\sum_{j=0}^{k-1}I_{n,j}}{\sum_{j=0}^{n}I_{n,j}}.
\end{align*}

\begin{proposition}
Almost surely under $\PP^* $ 
\[
G_n=\frac{1}{n}\sum_{k=1}^n\delta_{\alpha_{n,k}} \xrightarrow[n\to\infty]{d} {\mathbf 1}_{ x\in [0,1]} dx \quad \text{when} \quad n\longrightarrow\infty 
\]
\end{proposition}

\begin{proof}
In this setting the boundary conditions are almost surely regularly varying with index  $ c=1$ and $ q= \bar{p} = \frac12.$ So the Proposition is a corollary of Theorem~\ref{thm:cv_p}. 
\end{proof}
Actually in this special instance we may give fluctuations. 
\begin{proposition}
For all $ 0 < t, $ almost surely under $\PP^* $  
\[ \sqrt{2n +1} \left (G_n[0,t] - t \right)  \Rightarrow  W^o_t
\] 
where $ W^o $ is a Brownian bridge and it is a convergence in the weak sense in Skorokhod space. 
\end{proposition} 

\begin{proof}
Since $ (I_{n,k})_{0\leq k \leq n} $ are independent $\Gamma_2, $ random variables 
$$ \frac{1}{\sum_{j=0}^{n}I_{n,j}}(I_{n,0},\dots,I_{n,n})$$ has a Dirichlet distribution 
with dimension $ n+ 1, \; \mathcal{D}(2,\dots,2).$ 

So, if $ \{U_k\}_{1\leq k \leq 2n+1} $ is a family of independent random variables uniformly distributed on $ [0,1], $ the distribution of  the order statistics $$ (U_{(2)},\; U_{(4)},\dots,U_{(2n)}) $$
has the same distribution as the internal point produced by a Dirichlet distribution $  \mathcal{D}(2,\dots,2).$ 

Hence almost surely under $\PP^* $ $ (\alpha_{n,1},\dots,\alpha_{n,n})\stackrel{d}{=}(U_{(2)},\dots,U_{(2n)}).$
Then 
\begin{align*}
    G_n([0,t])&=\frac{1}{n}\sum_{k=1}^n\mathbf{1}(\alpha_{n,k}\leq t)
    \\&\stackrel{d}{=}\frac{1}{n}\sum_{k=1}^n\mathbf{1}(U_{(2k)}\leq t)
    \\&=\frac{1}{n}\sum_{k=1}^n\mathbf{1}((2n+1)F_n(t)\geq2k,
\end{align*}
where $  F_n(t)=\frac{1}{2n+1}\sum_{k=1}^{2n+1}\mathbf{1}(U_k\leq t). $
\begin{align*}
     G_n([0,t] &=\frac{1}{n}\max\{1\leq k \leq n:(2n+1)F_n(t)\geq2k\}
    \\&=\frac{1}{n}\left[\frac{(2n+1)F_n(t)}{2}\right].
\end{align*}
So $ \sup_{ t \in [0,1]}| G_n([0,t]) - F_n(t) | < \frac{1}{n}.$ 
Classically 
\[ \sqrt{2n +1} \left (F_n(t) - t \right)  \Rightarrow  W^o_t,
\]
see for instance Theorem 14.3 in ~\cite{Billingsley99}. Hence the Proposition is  a consequence of Slutsky Lemma. 
\end{proof}
Please note the contrast with fluctuations results in ~\cite{Cohen25} where the limits are deterministic and smooth. 
\section{Fragmentation for point processes}
\label{sec:pp}
In  this section a fragmentation procedure is introduced for  partitions of the full real line. 
To define such a partition we consider a sequence $ (x_{k})_{ k\in \ZZ}$ such that 
\begin{enumerate}
\item it is increasing \begin{equation}
\label{eq:inc} \forall k \in \ZZ, \; x_k < x_{k+1};
\end{equation}
\item its range is the full real line $ \RR$ 
\begin{equation}
\label{eq:span}  \lim_{ k \to - \infty} x_{k}= - \infty, \;  \lim_{ k \to + \infty} x_{k}= + \infty
\end{equation}
\item and we have the following convention for indexation~: 
\begin{equation}
\label{eq:index} x_0 \le 0 < x_1. 
\end{equation}
\end{enumerate}
The set of such sequences will be denoted by $ \cal S.$ 

For $n=0$ we start with $\calP_0=\{(a_{0,k},a_{0,k+1}], \; \forall k \in \ZZ \}$ where the sequence
$(a_{0,k})_{ k\in \ZZ} \in  \cal S.$ 

Please note that there is one to one correspondence between such sequences $(x_{k})_{ k\in \ZZ} $ and the simple integer valued measures on $ \RR $ $$ \mu = \sum_{ k\in \ZZ} \delta_{x_{k}}.$$ 
\begin{definition}
\label{def:frag-pp}
For a given family of splitting proportions 
$
\pmb p=(p_{n,k}:n\ge1,\, k\in \ZZ)
$ 
with $p_{n,k}\in(0,1)$ for all $k$, we  can  split and merge any $(a_{n-1,k})_{ k\in \ZZ} \in  \cal S $ as follows 
\begin{equation}
\label{eqind-pp}
\tilde{a}_{n,k}=p_{n,k}a_{n-1,k-1}+(1-p_{n,k})a_{n-1,k},\q k \in \ZZ.
\end{equation}
If $ \tilde{a}_{n,1}=p_{n,1}a_{n-1,0}+(1-p_{n,1})a_{n-1,1} > 0, $ we set $ a_{n,k} = \tilde{a}_{n,k} $ else  $ a_{n,k} = \tilde{a}_{n,k+1} $ to have a consistent  indexation. 
One can prove by induction that $(a_{n,k})_{ k\in \ZZ} $ belongs to  $\cal S .$ The transformation
 $$ (a_{n-1,k})_{ k\in \ZZ} \mapsto (a_{n,k})_{ k\in \ZZ} $$
 is called a fragmentation procedure.
\end{definition}
In this section we only consider the full random case where the $(P_{n,k}:n\ge1,\, k \in \ZZ)$ are  i.i.d. uniform \@ random variables in $[0,1].$ The sequence $(\calP_n)_{n\ge0}$ defined by the fragmentation procedure  forms a Markov chain whose state space is $ \mathcal S.$ We are interested in the behavior of 
\begin{equation}
\label{eq:mun}
\mu_n= \sum_{k \in \ZZ} \delta_{A_{n,k}},
\end{equation}
which is the counterpart of   $ n \tilde{G}_n $ in the previous sections. 

By extrapolating the construction of the previous section we can guess  invariant distributions of the Markov chain. Actually it turns out that the increments of the $ A_{n,k} $ will be i.i.d.  $\; \Gamma_2$ distributed. Following~\cite{Daley-v1-2003} one can define a stationary point process on $ \RR $ 
starting from a delayed renewal process on $ \RR^+ $ as follows. 

Let us consider $ (T_k)_{k \in \ZZ} $  a sequence 
of  i.i.d.  $\Gamma_2$ random variables. Let 
$ (X_k)_{k \in \NN} $ be defined as $ X_1 =T'_1 $ and for  $ k \ge 2,$ $ X_k = X_{k-1} + T_k.$
The process is delayed because the distribution of $ T'_1 $ is not $ \Gamma_2. $ It turns out that to define a stationary point process on $ \RR $ the distribution of $ T'_1 $ has to be a mixing of a $\Gamma_2 $ and an exponential random variable. Let us consider a Bernoulli random variable $ \varepsilon $ with parameter $ 1/2 $ independent of a pair of independent random variables $ \Gamma, E $ whose distributions are respectively $ \Gamma_2 $ and exponential with parameter $1.$ Then we set 
$  T'_1 = \varepsilon \Gamma + (1 - \varepsilon) E.$ It follows from Proposition 4.2.I in~\cite{Daley-v1-2003} that $ (X_k)_{k \in \NN} $ is stationary with the distribution of the forward recurrence time  
$ \forall u >0, \; X_u = \inf \{ X_k > u \} - u $ equal to the  distribution of $ T'_1.$ The $ X_k $'s will be the positive points of the stationary point process we want to construct on $ \RR.$ To define the negative points $  (X_k)_{k \le  0} $  we consider  $  T''_1 = (1 - \varepsilon) \Gamma +  \varepsilon E.$ Please note that $  T''_1  +  T'_1 = \Gamma + E $ and has a $\Gamma_3 $ distribution.  Let $ (X_k)_{k \le 0} $ be defined as $ X_0 =-T''_1 $ and for  $  k < 0,$ $ X_k = X_{k+1} - T_k.$ 

\begin{definition}
\label{NGamma2}
If we denote by $ \mu_{\Gamma_2} = \sum_{k \in \ZZ} \delta_{X_{k}}, $ $\mu_{\Gamma_2}  $ is a simple integer valued measure on $ \RR, $ which corresponds to a stationary point process $ N_{\Gamma_2} $  on $ \RR.$
\end{definition}

 Moreover the distribution of $ (X_0,X_1) $ is the distribution of $ (-T''_1,T_1) $ and the increments 
$ (X_k - X_{k-1})_{ k \neq 1} $ are i.i.d.  $ \Gamma_2.$ Let us also notice that the sequence $ (X_k)_{k \in \ZZ} \in \cal S. $

 Please remark that the distribution of  $ N_{\Gamma_2} $ can be obtained from a Poisson point process $ N $ on $  \RR $ and $ \epsilon $ a Bernoulli random variable with parameter $  1/2 $ independent of $N.$  If $ N = \sum_{k \in \ZZ} \delta_{\xi_{k}},$  where $\xi \in \cal S,$ the process
$$ N = \sum_{k \in \ZZ} \delta_{\xi_{2 k + \epsilon}} $$  has the distribution of $ N_{\Gamma_2}.$

\begin{proposition}
\label{prop:Gamma2-renew}
The distribution of the point process $ N_{\Gamma_2} $ is invariant for the fragmentation procedure  of Definition~\ref{def:frag-pp} with the splitting proportions $ (P_{1,k}: \, k \in \ZZ)$ independent of the $X_k$'s.
\end{proposition}
\begin{proof}
 If $ \mu_{\Gamma_2} = \sum_{k \in \ZZ} \delta_{X_{k}}, $ we take for $\calP_0=\{(X_{k-1},X_{k}], \; \forall k \in \ZZ \}$ and we consider $\calP_1=\{(Y_{k-1},Y_{k}], \; \forall k \in \ZZ \}$ defined by the fragmentation procedure  of Definition~\ref{def:frag-pp} with the splitting proportions $ (P_{1,k}: \, k \in \ZZ)$ independent of the $X_k$'s. First recall that the increments $ (X_k - X_{k-1})_{ k \in \ZZ} $ are split by fragmentation into two subintervals, with proportions $1-P_{1,k}$ on the left and $P_{1,k}$ on the right. Let us denote by $ \Delta_{1,2k} = (1-P_{1,k}) (X_k - X_{k-1}) $ the length of the left  part and $ \Delta_{1,2k+1} = P_{1,k} (X_k - X_{k-1}). $ Since $ (X_k - X_{k-1})_{ k \neq 1} $ are i.i.d.  $ \Gamma_2 $ the $ (\Delta_{1,k})_{ k \neq 2,3} $ are i.i.d.   exponential with parameter one. It implies that $ (Y_k - Y_{k-1})_{ k \neq 0,1,2} $ are i.i.d.  $ \Gamma_2.$ 

We have know to investigate how the interval $ (X_0,X_1]$ splits. It depends on the sign of 
$ \tilde{X}_{1}=P_{1,1} X_{0}+(1-P_{1,1})X_1. $ If $ \tilde{X}_{1} > 0,$ then $ Y_1 = \tilde{X}_{1} $ and 
$ Y_0 = - X_0 -  \Delta_{1,1}.$  If $ \tilde{X}_{1} < 0,$ then $ Y_1 = X_1 +   \Delta_{1,4} $ and 
$ Y_0 = \tilde{X}_{1}. $ Hence the pair $(-Y_0,Y_1) $ is the value at time 1 of a Markov chain $ \mathcal X^1 $ valued in $(0,+\infty)^2$ starting from  $(-X_0,X_1). $ Let us first define  $ \mathcal X^1 $  
and then we will show that the distribution of $ (T''_1,T'_1)$ is an invariant distribution for $ \mathcal X^1.$ 

\begin{definition} 
\label{def:MarkovR2}
Let us consider $U$ a random variable uniformly distributed on $ (0,1) $  and independent of a random variable $ E $ with exponential distribution with parameter $1.$ Let us define the Markov chain $ \mathcal X^1 $ on $(0,+\infty)^2$ as follows. If $ \mathcal X^1 (0)=(b,t) \in (0,+\infty)^2$ and, if $ - U b + (1-U) t >0,$ $$ \mathcal X^1(1) = (b+ E,- U b + (1-U) t). $$ Else $$ \mathcal X^1(1) = (- U b + (1-U) t,t+E). $$ 
\end{definition}

Let us now prove the following Lemma. 
\begin{lemma}
\label{lem:inv-MarkovR2}
The distribution of $ (T''_1,T'_1)$ is an invariant distribution for $ \mathcal X^1.$
\end{lemma}
\begin{proof}
One can first remark that the distribution of $ (T''_1,T'_1)$ has a density $$ \frac{e^{-(b+t)}}{2(b+t)} $$ with respect to the Lebesgue measure on $(0,+\infty)^2.$ Let us denote by $ \tilde{X}_1 = - U T''_1 + (1-U) T'_1,$ elementary computations yield that the conditional distribution of $ (T''_1,\tilde{X}_1 )  $ given  $ \tilde{X}_1 > 0 $ is the pair of independent exponential distributions with parameter $ 1.$ Therefore  the conditional distribution of $ \mathcal X^1(1) $ given  $ \tilde{X}_1 > 0 $ is the pair of a $\Gamma_2 $ distribution and an independent exponential distribution with parameter $ 1.$ 
Given $ \tilde{X}_1 <  0 $ the conditional distribution of $ (-\tilde{X}_1, T'_1 )  $ is the pair of independent exponential distributions with parameter $ 1.$ Therefore  the conditional distribution of $ \mathcal X^1(1) $ given  $ \tilde{X}_1 < 0 $ is the pair of an  exponential distribution with parameter $ 1$ and an independent $\Gamma_2 $ distribution. Since the probability that $ \tilde{X}_1 > 0 $  is one half the Lemma is proved. 

\end{proof}

 Because of Lemma~\ref{lem:inv-MarkovR2} the distribution of $(Y_0,Y_1)$ is the same as the distribution of $ (X_0,X_1).$ Moreover the distribution of $ X_1 - \tilde{X}_1  $ given $ \tilde{X}_1 > 0 $ 
is  exponential and  independent of $  \Delta_{1,4}.$ Hence given  $ \tilde{X}_1 > 0 $ 
$ Y_2 - Y_1 $ has a $\Gamma_2 $ distribution. Given  $ \tilde{X}_1 < 0,$ we also get that $ Y_0 - Y_{-1} $ and $ Y_2 - Y_1 $ have  independent $\Gamma_2$ distribution. It follows that $ (Y_k - Y_{k-1})_{ k \neq 1} $ are i.i.d.  $ \Gamma_2 $ distributed  and  $\calP_1 $ has the same distribution as $\calP_0. $
\end{proof}
\begin{remark} With the same proof we can show that,  for any $ \theta>0, $ if we take increments with $\Gamma_{2,\theta} $ and $ E $ an exponential of parameter $ \theta > 0 $ we can construct point processes that are also invariant for the fragmentation procedure. 
\end{remark}

 We want to prove that the limit of the special setup of the section~\ref{sec:rand-expan} is  the point process  $N_{\Gamma_2}.$ We refer to~\cite{Kallenberg17} for the meaning of the convergence of point processes in distribution. Let us first consider $ \tilde{\mu}_n = \sum_{k=1}^n \delta_{A_{n,k}} $ of~\eqref{eq:rand-mea} in section~\ref{sec:rand-expan}. Since almost surely $ S_n \to \infty $ and $ S'_n \to -\infty ,$ for any 
$ -\infty < l < m  < +\infty $ there exists $n_0 $ such that for $ n \ge n_0, S'_n< l <m < S_n $ almost surely. Hence we aim to apply (ii) Theorem 4.11 of ~\cite{Kallenberg17} for $ \tilde{\mu}_n $ and $ \mu_{\Gamma_2}.$ 
More precisely we have to show that for any non-negative continuous function $ f $  with compact support in 
$\RR$ $ \tilde{\mu}_n f $ converges in distribution to $  \mu_{\Gamma_2} f .$ To show this convergence we introduce a generalization of the Markov chain $ \mathcal X^1. $ The convergence $ \nu_n f \to \mu_{\Gamma_2} f $ will be a consequence of the geometric ergodicity of a new Markov chain. 

 Let us first define this Markov chain that we will denote by $ \mathcal X. $ The Markov chain will be  valued in the following space.
 \begin{definition}
 \label{def:Slm}
 For any let $ -\infty < l < m  < +\infty,$ and let $ {\cal S}_{l,m} $ be the space of finite increasing 
 sequences $ x^1< l < x^2< \ldots < x^{n-1} < m < x^n $ with at least two points ($ n \ge 2.$)
 \end{definition}
 We fix $  -\infty < l < m  < +\infty,$  until further notice. 
 \begin{definition}
 \label{def:markovNN}
 Let $(U_{i,k})_{ k\in \ZZ, \;  i \in \NN} $ be a sequence of i.i.d. uniform  random variables on $[0,1]$ independent of $(E_{1,i}, \; E_{2,i})_{ i \in \NN} $ be a pair of independent sequences of i.i.d. exponential  random variables with parameter $1.$ Both sequences are assumed to be independent. Let us define a $ {\cal S}_{l,m} $ valued Markov chain with the following procedure. 
 
 For $ i \ge 0 $ let $ \X_i = \{ \X_{i}^{1}< l < \X_{i}^{2}< \ldots < \X_{i}^{n-1} < m < \X_{i}^{n} \}.$
 \begin{itemize}
 \item  If   $ n > 2.$ Then for $ k = 2 $ to $ n $ let 
 $$ \tilde{\X}_i^k = U_{i+1,k} \X_{i}^{k-1} + (1-  U_{i+1,k}) \X_{i}^{k}.$$
 \begin{itemize}
 \item  If $ \tilde{\X}_i^2 < l,$ we set $ \X_{i+1}^1 = \tilde{\X}_i^2,$ and for $ k = 2 $ to $ n-1 $ $ \X_{i+1}^{k} =  \tilde{\X}_i^{k+1}.$ 
 \item  If  $ \tilde{\X}_i^2  > l,$ we take $ \X_{i+1}^1 =  \X_{i}^1 - E_{1,i}  $ and for  $ k = 2$ to $ n-1 $  $ \X_{i+1}^{k} =  \tilde{\X}_i^k.$ 
 \item  If $ \tilde{\X}_i^n < m , \; \X_{i+1}^n=\tilde{\X}_i^n $ and $ \X_{i+1}^{n+1}= \X_i^n +  E_{2,i}.$      \item If $ m < \tilde{\X}_i^n, \;  \X_{i+1}^n= \tilde{\X}_i^n.$
 \end{itemize}
 \item If $ n = 2,$ 
 $$ \tilde{\X}_i^2 = U_{i+1,2} \X_{i}^{1} + (1-  U_{i+1,2}) \X_{i}^{2}.$$
 \begin{itemize}
 \item  If $ \tilde{\X}_i^2 < l,$ we set $ \X_{i+1}^1 = \tilde{\X}_i^2,$ and $ \X_{i+1}^2 = \X_i^2+ E_{2,i}.$
 \item  If $ l < \tilde{\X}_i^2 < m,$ we set 
 $$   \X_{i+1}^1 =  \X_{i}^1 - E_{1,i}, \;  \X_{i+1}^2 =  \tilde{\X}_i^2, \; \X_{i+1}^3 = \X_i^2 +  E_{2,i}. $$
 \item  If $ m < \tilde{\X}_i^2,$ we set $ \X_{i+1}^1 = \tilde{\X}_i^1 - E_{1,i},$ and $ \X_{i+1}^2 = \tilde{\X}_i^2.$
 \end{itemize}
 \end{itemize} 
\end{definition}

 Please note that the number of points in $ \X $ is random. For instance the number of points $ \X_{i+1} $
 is the number of points $ \X_{i} $ plus one on the event $ \tilde{\X}_i^2  > l $ and $ \tilde{\X}_i^n < m .$ 

 The Markov chain $\X $ mimics the behavior of  the point process $N_{\Gamma_2}  $ around the interval $[l,m].$ Let us define this formally.
 \begin{definition}
 \label{def:around}
 If $\mu_{\Gamma_2} = \sum_{ k \in \mathbb Z} \delta_{X_k}, $ let us denote by $ l_0 = \max \{k: X_k < l \} $ 
 and $ m_0 = \min\{k: m < X_k \}, $ then the point process $N_{\Gamma_2}  $ around the interval $[l,m] $
 is the finite sequence 
  $$ X_{l_0} <   X_{l_0+1}< \ldots< X_{m_0-1}  < X_{m_0}.$$
 \end{definition}
 Then it is easy to prove the following proposition.
 
 \begin{proposition}
 \label{prop:markov}
 The fragmentation procedure of Definition~\ref{def:frag-pp} applied to the point process $N_{\Gamma_2}  $ around the interval $ [l,m] $ has the same transition as $\X.$ 
 \end{proposition}
 

 Similarly  the points of the random measure $(\tilde{\mu}_{n+n_0})_{n \ge 0}$ of section~\ref{sec:rand-expan} around the interval $[l,m] $ have the same transition as $\X.$ Hence to prove the convergence $\tilde{\mu}$ to $ \mu_{\Gamma_2} $ around the interval $[l,m] $ we will show the geometric ergodicty of  $\X.$ We rely on~\cite{Meyn2009} for Markov chain convergence on general state space. 
 
 \begin{proposition}
 \label{prop:geom-erg}The Markov chain $\X$  is $\phi$-irreducible, aperiodic  and for $ M > 0  $ big enough,  $ C= \{x \in {\cal S}, \;  l - x^1 \le M (m-l) \} $ is a petite set,
 \begin{equation}
 \label{eq:V}
V(x) = l - x^1 + 1.
\end{equation}  is a Lyapounov function satisfying the 
 $ V2 $ property of~\cite{Meyn2009}. Consequently $\X$   is positive Harris recurrent and is  geometric ergodic in the sense of Theorem~15.0.1 of~\cite{Meyn2009}.
 \end{proposition}
 \begin{proof}
 The set $\cal S $ is a disjoint union of subsets  ${\cal S}_n = \{ x \; \text{s.t.} \; x^1< l < x^2< \ldots < x^{n-1} < m < x^n \} $ of  $\RR^n$ which we can endow with the Lebesgue $\lambda_n $ on $\RR^n.$ If we take the measure $\phi = \sum_{n=2}^{\infty} \lambda_n, $ $ \cal S $ is $\phi$-irreducible. 
 
  For $ M >0,$ the drift of $ V $~:  $$\Delta V(x) \DP \EE (V(\X_1 | \X_{0} = x ) - V(x) $$ satisfies
 \begin{equation}
 \label{eq:drift-V} 
 \Delta V (x) \le -\frac{l-x^1}{2} \frac{M}{1+M} + \frac{M+2}{1+M},
 \end{equation}
 if $ l - x^1  >  M (m-l).$ 
 Then we choose $ M $ big enough such that there exists $ b \in \RR,$ for which  $ V $ satisfies
 \begin{equation}
 \label{eq:drift-V2} 
 \Delta V (x) \le - 1  + b \mathbf{1}_C (x), 
 \end{equation}
 which is V2 assumption, 
 and 
 \begin{equation}
 \label{eq:drift-15-3} 
 \Delta V (x) \le - \frac{1}{3}  V(x) + b \mathbf{1}_C (x) 
 \end{equation}
 which is the inequality (15.3) of Theorem 15.0.1 of~\cite{Meyn2009}.
 Then because of Theorem~11.3.4,  $ C $ is a petite set  and $ \X$ is positive  Harris recurrent. 
 Moreover by Theorem 15.0.1 $\X$ is geometric ergodic, if we can show that $ \X$ is aperiodic. 
 
  To show the aperiodicity we consider the transition kernel $ p $ of $\X $ on the set $ \tilde{C} = \{x \in {\cal S}_2 \; \text{s.t.} \;   l - x^1, \; x^2 -m \in [\frac{1}{M}, M]  \} $ for $ M>0.$ Actually we will show that $  \tilde{C} $ is a $\phi_2 $ and $\phi_3 $ (with the notation of~(5.14) in~\cite{Meyn2009} for the minorizing measures)  small set. If $ x \in {\cal S}_2 $ the probability that $\X $ jumps from $ x $ to another point in  $ {\cal S}_2 $ is $$ \frac{x^2-x^1 - (m-l)}{x^2 - x^1}.$$ Conditionally to this event the transition probability for   $ x, x' \in {\cal} S_2 $ is 
  \begin{equation}
  \label{eq:trans}
  p(x^1,x^2,x'^{1},x'^{2}) = \frac{\mathbf{1}_{0 <x'^{1} < x^1, \, x'^{2} > x^2 } e^{-(x'^{2}-x^2)} + \mathbf{1}_{ x'^{1} > x^1, 0 <x'^{2} < x^2 } e^{-(x'^{1}-x^1)}}{x^1+x^2}.
   \end{equation} It yields for the transition probability of the $2$-step Markov chain denoted by $p _2$ 
   \begin{align*}
   p_2(x^1,x^2,x'^{1},x'^{2}) &\ge \int_{[M,M+1] \times [0,1/M]} \frac{e^{-(b-x^1)}}{x^1+x^2}\frac{e^{-(x'^{2}-t)}}{b+t}\,db\, dt 
   \\
   &\ge \frac{e^{-(M+1)} e^{-M}}{2 M^2 (M+1+1/M)}  
   \end{align*}
 which implies that  $ \tilde{C}  $ is a small set for  the $2$-step Markov chain. Then the $3$-step Markov chain has transition 
 $$ p_3(x^1,x^2,x'^{1},x'^{2}) =  \int_{0}^{+ \infty} \int_{0}^{+ \infty} p_2(x^1,x^2,b,t)  p(b,t,x'^{1},x'^{2}) \,db\, dt. $$ For $ 0 < \delta < 1/M, $  we consider $ {\cal R}_1 = [\delta, 1/M] \times [M,M+1] $ and  $ {\cal R}_2 =  [M,M+1]  \times [\delta, 1/M].$  For $ x, x' \in  \tilde{C}   $
\begin{align*}
 p_3(x^1,x^2,x'^{1},x'^{2})&=  \int_{{\cal R}_1}\int_{{\cal R}_2} p(x^{1},x^{2},b,t) p(b,t,b',t') p(b',t',x'^{1},x'^{2}) \, db \,dt db'\,dt' \\
 & \ge  \int_{\delta}^{ 1/M}\int_{M}^{M+1}\int_{M}^{M+1} \int_{\delta}^{ 1/M}  \frac{e^{-(M+1)}}{x^1+x^2}  \frac{e^{-(M+1)}}{b+t} \frac{e^{-M}}{b'+t'} \,db \,dt \,db'\,dt' \\
 & \ge \frac{e^{-(3M+2)}}{2 M (M+1+1/M)^2} \left (\frac{1}{M}-\delta \right )^2.
  \end{align*}
  Hence $  \tilde{C} $ is a $\phi_2 $ and $\phi_3 $ small set, which implies aperiodicity. The proof of the proposition is done. 
 \end{proof}
 
 Now we conclude this section by proving the convergence result we outlined before Proposition~\ref{prop:geom-erg}. 
 \begin{theorem}
 \label{th:vag-conv} The random measure  $(\tilde{\mu}_{n})_{n \ge 0}$ of~\eqref{eq:rand-mea} in  section~\ref{sec:rand-expan} converges in distribution  to the random measure $  \mu_{\Gamma_2} $  vaguely
 \end{theorem}
\begin{proof}
We want to apply  (iii) of Theorem~4.11 in~\cite{Kallenberg17}. Here the space $\hat{C}_S $ of~\cite{Kallenberg17} is the set of continuous, positive functions with compact support in $ \RR $   
and (iii) means that $ \forall f \in \hat{C}_S, \; \tilde{\mu}_n f = \sum_{k \in \ZZ} f(A_{n,k}), $ where the $ A_{n,k} $ are defined in~\eqref{def:point-rand-expan}, converges in distribution to $  \sum_{k \in \ZZ} f(X_k) $ where the $X_k$'s are defined in Definition~\ref{NGamma2}. Let $  -\infty < l < m  < +\infty,$   be chosen such that the support of $ f $ is included in  $ [l,m],$ we consider the functional $ \pi_f $ on $ {\cal S} $ defined by $\pi_f = \sum_{i=1}^n f(\X_i).$ Because of Proposition~\ref{prop:markov} The point process $N_{\Gamma_2}  $  around  the interval $[l,m] $  has the Markov transition of $ \X.$ Similarly the induction~\eqref{def:point-rand-expan} applied to the points of the random measure $(\tilde{\mu}{n+n_0})_{n \ge 0}$ of section~\ref{sec:rand-expan} around the interval $[l,m] $ have the same transition as $\X.$ The functional $ \pi_f $ is continuous on $ {\cal S},$ hence the geometric ergodicity of $\X$ yields the convergence in distribution of $ \tilde{\mu}_n f $ toward  $ \mu_{\Gamma_2} f.$ 
\end{proof}

\section{Acknowledgment}
 The authors would like to thank Michel Pain for fruitful discussions and Institut Mathématiques de Toulouse for the support of Mr. Shambo during his stay in Toulouse. 

\def\cprime{$'$}


\def\cprime{$'$}
\begin{thebibliography}{1}

\bibitem{Billingsley99}
P. Billingsley.
\newblock {\em Convergence of probability measures}.
\newblock John Wiley \& Sons Inc., New York, 1999.
\newblock 2nd Edition.

\bibitem{Bingham89}
N.~H. Bingham, C.~M. Goldie, and J.~L. Teugels.
\newblock {\em Regular variation}, volume~27 of {\em Encyclopedia of
  Mathematics and its Applications}.
\newblock Cambridge University Press, Cambridge, 1989.

\bibitem{Cohen25}
S. Cohen, J. Norris, M. Pain, and G. Samorodnitsky.
\newblock {Interlacing sequnces resulting from an interval split-merge dynamics
  and the induced probability}.
\newblock https://hal.science/hal-04947651 , to appear in ALEA, February 2025.

\bibitem{Daley-v1-2003}
D.~J. Daley and D.~Vere-Jones.
\newblock {\em An introduction to the theory of point processes. {V}ol. {I}}.
\newblock Probability and its Applications (New York). Springer-Verlag, New
  York, second edition, 2003.
\newblock Elementary theory and methods.

\bibitem{Feller50}
W~Feller.
\newblock {\em An Introduction to Probability Theory and its Applications},
  volume~2.
\newblock Wiley, 1950.

\bibitem{Kallenberg17}
O. Kallenberg.
\newblock {\em Random measures, theory and applications}, volume~77 of {\em
  Probability Theory and Stochastic Modelling}.
\newblock Springer, Cham, 2017.

\bibitem{Meyn2009}
S. Meyn and R.~L. Tweedie.
\newblock {\em Markov chains and stochastic stability}.
\newblock Cambridge University Press, Cambridge, second edition, 2009.
\newblock With a prologue by Peter W. Glynn.

\end{thebibliography}
\end{document}